\documentclass[journal]{IEEEtran}
\usepackage[latin1]{inputenc}
\usepackage[cmex10]{amsmath}
\usepackage{amsfonts}
\usepackage{amssymb}
\usepackage{graphicx}
\usepackage{verbatim}
\usepackage{array}
\usepackage{multirow}
\usepackage{dcolumn}
\usepackage{color}
\usepackage[noadjust]{cite}
\usepackage{url}
\usepackage{balance}
\usepackage{booktabs}
\usepackage[mathscr]{euscript}
\DeclareGraphicsExtensions{.eps}
\usepackage{setspace}
\usepackage{graphicx,epstopdf}
\usepackage{comment}
\usepackage{enumerate}
\usepackage{flushend}
\usepackage[dvipsnames]{xcolor}

\usepackage{bbold}
\usepackage{fancyhdr,graphicx,amsmath,amssymb}
\usepackage{mathtools}

\usepackage{amsthm}

\usepackage{makecell}
\usepackage{tabularx}
\usepackage{tablefootnote}
\usepackage{threeparttablex}
\usepackage{float}

\newtheorem{theorem}{Theorem}[section]

\newtheorem{prop}[theorem]{Proposition}

\usepackage{geometry}
\begin{document}

\title{Boundary Technology Costs for Economic Viability of Long-Duration Energy Storage Systems }

\author{ 
Patricia Silva, \IEEEmembership{Student Member,~IEEE,}
Alexandre Moreira, \IEEEmembership{Member,~IEEE,} Miguel Heleno, \IEEEmembership{Senior Member,~IEEE,} and Andr\'e~Lu\'is~Marques~Marcato,~\IEEEmembership{Senior Member,~IEEE}

}

\maketitle
\begin{abstract} The urgent need for decarbonization in the energy sector has led to an increased emphasis on the integration of renewable energy sources, such as wind and solar, into power grids. While these resources offer significant environmental benefits, they also introduce challenges related to intermittency and variability. Long-duration energy storage (LDES) technologies have emerged as a very promising solution to address these challenges by storing excess energy during periods of high generation and delivering it when demand is high or renewable resources are scarce for a sustained amount of time. This paper introduces a novel methodology for estimating the boundary technology cost of LDES systems for economic viability in decarbonized energy systems. Our methodology is applied to estimate the boundary costs in 2050 for the state of California to achieve full retirement of gas power plants. California's ambitious decarbonization goals and transition to a renewable energy-based power system present an ideal context for examining the role of LDES. The results also offer insights into the needed capacity expansion planning and the operational contribution of LDES in the California's energy landscape, taking into account the unique energy demand profiles and renewable resource availability of the region. Our findings are intended to provide complementary information to guide decision-makers, energy planners, and any other stakeholders in making informed choices about LDES investment in the context of a decarbonized energy future.
\end{abstract}


\section*{Nomenclature}\label{Nomenclature}

\subsection*{Sets}

\begin{description} [\IEEEsetlabelwidth{100000000}\IEEEusemathlabelsep]
	
    \item[$G$] Set of indexes of all generators.

    \item[$G^{cand}$] Set of indexes of generators that are candidate for investment.

    \item[$G^{fixed}$] Set of indexes of fixed existing generators.

    \item[$G^{firm}$] Set of indexes of generators able to provide firm dispatchable generation.

    \item[$G^{renew}$] Set of indexes of renewable generators.

    \item[$G^{firm,fixed}$] Set of indexes of generators equivalent to $G^{firm} \cap G^{fixed}$.

    \item[$G^{renew,fixed}$] Set of indexes of generators equivalent to $G^{renew} \cap G^{fixed}$.

    \item[$G^{renew,cand}$] Set of indexes of generators equivalent to $G^{renew} \cap G^{cand}$.

    \item[$G{\color{black}^{res,providers}}$] Set of indexes of generators that can provide reserve.

    \item[$H$] Set of indexes of all energy storage systems.

    \item[$H^{cand}$] Set of indexes of storage systems that are candidate for investment.

    \item[$H^{fixed}$] Set of indexes of fixed existing storage systems.

    \item[$H^{long}$] Set of indexes of long-duration storage systems.

    \item[$H^{short}$] Set of indexes of short-duration storage systems.

    \item[$H^{short,fixed}$] Set of indexes of storage systems equivalent to $H^{short} \cap H^{fixed}$.

    \item[$H^{long,cand}$] Set of indexes of storage systems equivalent to $H^{long} \cap H^{cand}$.

    \item[$H^{short,cand}$] Set of indexes of storage systems equivalent to $H^{short} \cap H^{cand}$.

    \item[$T$] Set of indexes of time periods.

\end{description}

\subsection*{Parameters}
\begin{description} [\IEEEsetlabelwidth{100000000}\IEEEusemathlabelsep]

    \item[$\eta_{h}$] Round trip efficiency of storage system $h$.
    
    \item[$C^{I}$] System power imbalance cost.

    \item[${\color{black}C^{short}}$] Reserve shortage cost.

    \item[$C^{inv,gen}_g$] Equivalent annual investment cost of candidate generator $g$.

    \item[$C^{fom,gen}_{g}$] Annual fixed operation and maintenance cost of generator $g$.

    \item[$C^{fom,st,power}_{h}$] Annual fixed operation and maintenance cost of storage system $h$.

    \item[$C^{p}_{gt}$] Generation cost of generator $g$.
    
    \item[$C^{st,energy}_h$] Equivalent annual investment cost in energy capacity for storage system $h$.

    \item[$C^{st,power}_h$] Equivalent annual investment cost in power capacity for storage system $h$.

    \item[$C^{up}_{gt}$] Reserve provision cost of generator $g$.

    \item[$f^{available}_{gt}$] Number between $0$ and $1$ that determines how much of the generation capacity of renewable unit $g$ is available during time $t$.

    \item [${\color{black}D_{t}}$] Demand of the system at time period $t$.
    
    \item[$\overline{P}_g$] Power generation capacity of generator $g$.

    \item[$\overline{P}^{st,power}_h$] Maximum power charge/discharge limit for existing storage system $h$.

    \item[$\overline{P}^{st,power,ini}_h$] Maximum initial power charge/discharge limit for candidate storage system $h$.

    \item[$\overline{R}^{up,factor}_g$] Number between $0$ and $1$ that determines how much of the generation capacity of unit $g$ can be used for reserves.

    \item[$r^{up,min}_{t}$] Minimum amount of reserve to be held by the system.

    \item[${RD}^{factor}_g$] Number between $0$ and $1$ that determines the ramp-down capability of unit $g$ relative to its generation capacity.

    \item[${RU}^{factor}_g$] Number between $0$ and $1$ that determines the ramp-up capability of unit $g$ relative to its generation capacity.

    \item[$\overline{S}_h$] Duration of storage system $h$.

    \item[$\underline{V}_h$] Minimum state of charge limit for storage system $h$.

    \item[$\overline{V}_h$] Maximum state of charge limit for storage system $h$.

    \item[$\overline{x}^{inv,gen}_g$] Maximum limit of generation capacity investment for candidate generator $g$.

    \item[$\underline{x}^{ret,gen}_g$] Number between $0$ and $1$ that determines the minimum reduction in the generation capacity of unit $g$.

    \item[$\overline{x}^{ret,gen}_g$] Number between $0$ and $1$ that determines the maximum reduction in the generation capacity of unit $g$.

    \item[$\overline{x}^{st,energy}_h$] Maximum limit of energy capacity investment for candidate storage system $h$.

    \item[$\overline{x}^{st,power}_h$] Maximum limit of power capacity investment for candidate storage system $h$.

    \item[$x^{st,power \ddagger}_{h}$] Predefined power capacity for storage $h$ to be considered in the {\it opportunity value maximization model}.

\end{description}

\subsection*{Decision variables}

\begin{description} [\IEEEsetlabelwidth{5000000}\IEEEusemathlabelsep]

    \item[$\Delta^-_t$] Negative power imbalance during time $t$.
    
    \item[$\Delta^+_t$] Positive power imbalance during time $t$.
    
    \item[$\delta^{up,short}_t$] Reserve shortage during time $t$.
    
    \item[$c^{BC}$] Boundary cost of LDES.
    
    \item[$p_{gt}$] Power generation of unit $g$ during time $t$.
    
    \item[$p^{st,ch}_{ht}$] Power charge of storage $h$ during time $t$.
    
    \item[$p^{st,dis}_{ht}$]Power discharge of storage $h$ during time $t$.
    \item[$\overline{p}^{rem}_{g}$] Remaining generation capacity of unit $g$ after reduction.
    
    \item[$q^{over}$] Budget overrun relative to the overall cost determined by the baseline model.
    
    \item[$r^{st,up}_{ht}$] Reserve provisioned by storage $h$ during time $t$.
    
    \item[$r^{up}_{gt}$] Reserve provisioned by generator $g$ during time $t$.
    
    \item[$v_{ht}$] State of charge of storage $h$ during time $t$.
    
    \item[$x^{inv,gen \dagger}_{g}$] Generation capacity of generator $g$ after investment decision.
    
    \item[$x^{ret,gen \dagger}_{g}$] Generation capacity of generator $g$ to be reduced after retirement decision.
    
    \item[$x^{st,energy \dagger}_{h}$] Energy capacity of storage system $h$ after investment decision.
    
    \item[$x^{st,power \dagger}_{h}$] Power capacity of storage system $h$ after investment decision.

\end{description}

\section{Introduction}

\IEEEPARstart{C}{limate} concerns motivated a necessary global movement towards decarbonization, with countries all over the world pledging to reduce carbon emissions and reach net-zero emissions in the following years \cite{Net_zero_by_2050}. In the United States, the emissions of greenhouse gases (GHG) related to energy are expected to decrease by 25\% to 38\% from their 2005 levels by 2030, according to the projections in the Annual Energy Outlook (AEO) 2023  \cite{AEO_2023}. 
In California, the targets are even more aggressive according to the 100 Percent Clean Energy Act of 2018, also known as Senate Bill 100 (SB 100). The SB 100 essentially establishes that: (i) renewable sources will supply 60\% of California's total energy demand by 2030 and (ii) carbon neutrality must be achieved by 2045 \cite{SB_1000}. Within this context, 
long-duration energy storage (LDES) systems can play a vital role in achieving these objectives and it is of utmost importance to estimate the technology costs at which the service provided by LDES will become economically viable to support a cost-effective decarbonization of the energy sector.





The integration of renewable energy sources (RES), such as solar and wind, is the main measure to meet the aforementioned decarbonization targets. The variability and intermittency of RES, however, impose significant challenges to power systems, which have been originally designed to operate with firm and dispatchable resources, including, for example, hydro power plants, natural gas units, and coal generators \cite{dowling2020role}. The fluctuations in the effective availability of RES, nonetheless, can be absorbed and alleviated by energy storage (ES) systems, which, in general, enable the temporal shifting of energy. As a consequence, research about ES systems has been gaining traction with works discussing their participation in different exercises involving, for example, electricity markets \cite{zheng2023energy, he2023energy, zuluaga2023data} and transmission expansion planning \cite{qiu2016stochastic}. In addition, as a matter of fact, ES systems are being installed in their short-duration (usually around 4 hours) version in real-world power systems to facilitate operations under the current levels of renewable integration \cite{eia_battery_storage}.

The fundamental participation of short-duration energy storage (SDES) systems in counterbalancing the fluctuation of renewables notwithstanding, these systems are limited to contribute in their discharging mode during a short period of time within a day and cannot support the operation of a power system throughout a potential sustained amount of time when the generation output of RES might be substantially low. In fact, firm generation conventional technologies are still currently needed to provide this long-duration type of support. As we move towards fully decarbonized power systems, these firm technologies will need to be retired and an alternative to compensate renewable intermittence will be necessary to achieve 100\% renewable integration. This alternative can be LDES systems due to their potential instrumental role in providing power discharge over a prolonged amount of time. 
However, at the moment, LDES systems are non-mature technologies whose costs are not well defined yet, which hinders a proper evaluation of their contribution to achieving policy goals such as decarbonization. Therefore, to adequately inform policies and programs that aim to make LDES economically viable (such as the DOE's Long Duration Storage Shot \cite{DOE_LDES_shot}), there is a strong need for a systematic methodology that can appraise target LDES technology costs that will enable a cost-effective and beneficial adoption of LDES systems.


\vspace{-0.2cm}
\subsection{Literature review}

Currently, there is no consensus on the best technology to provide LDES services, with many different options to play this role under development \cite{Driving_to_Net_Zero_Industry_2023}, and it is unclear what the thresholds of costs and specifications should be for significant adoption  \cite{sepulveda2021design}. 
For instance, lithium-ion is currently the most popular storage technology and it is able to provide a sustained output over a long period of time \cite{sepulveda2021design}. However, lithium-ion has a high cost per kWh of energy storage capacity, which hinders its scalability for long-duration storage applications \cite{albertus2020long}. In addition, the ideal duration of an LDES system is still a matter of discussion. Most works consider that LDES has a minimum duration of 10 hours \cite{albertus2020long, zhang2020benefit, dowling2020role, Liftoff_DOE_2023}. The US Department of Energy (US-DOE), in its turn, splits LDES technologies according to the following duration categories: (i) inter-day LDES (10--36 hours), (ii) multi-day/week LDES (36--160 hours), (iii) and seasonal shifting (160+ hours) \cite{Liftoff_DOE_2023}. Within these different categories, several technologies of LDES are under development with diverse cost and performance characteristics. 

{\color{black}LDES technologies can be divided into four main types, namely, mechanical, chemical, electrochemical, and thermal \cite{Driving_to_Net_Zero_Industry_2023}. Some examples include: (i) mechanical (pumped hydroelectric storage (PHS) and compressed air energy storage (CAES)), (ii) electrochemical (zinc or vanadium flow batteries, lithium-ion, sodium, and iron-air batteries), (iii) thermal (concentrating solar power (CSP)), (iv) chemical (hydrogen storage). }
%
%
%
Furthermore, hybrid systems can also provide LDES services. For example, in \cite{vecchi2023long}, the authors explored six Thermo-Mechanical Energy Storage (TMES) technologies: adiabatic compressed air energy storage (ACAES), liquid air energy storage (LAES), pumped thermal energy storage (PTES), oxides energy storage (OES), carbonates energy storage (CES), and hydroxides energy storage (HES).

%
%
%
%

Due to its potential, the participation of LDES in capacity planning studies has been attracting an increasing deal of attention in the technical literature. For instance, in \cite{ziegler2019storage}, the authors optimize the mix of investments in wind and solar generators in Arizona, Iowa, Massachusetts, and Texas for different energy and power capacity costs of storage systems while also varying their duration and indicate US\$1000 $\text{kW}^{-1}$ and US\$20 $\text{kWh}^{-1}$ as competitive technology costs for LDES. By exploring the general LDES cost-performance parameter space using a discounted cash flow framework, the authors of \cite{albertus2020long} conclude that technology costs of energy capacity must be substantially reduced to US\$3 $\text{kWh}^{-1}$ for a 100h duration ES, and US\$7 $\text{kWh}^{-1}$ for a 50h duration ES, which would allow technology costs of power capacity to vary from US\$500 $\text{kW}^{-1}$ to US\$1000 $\text{kW}^{-1}$, assuming at least 50\% of RTE. According to the capacity expansion studies presented in \cite{hargreaves2020long}, ES systems with a duration ranging between 10 and 100 hours can become competitive at a marginal technology cost varying within US\$2.5--20 $\text{kWh}^{-1}$. {\color{black} Another capacity expansion study in \cite{sepulveda2021design} tested different values of technology costs and efficiencies for LDES and indicated 
maximum values of US\$20 $\text{kWh}^{-1}$ and US\$1400 $\text{kW}^{-1}$ as well as an RTE of 72\% to make LDES technologies able to reduce 
system costs by at least 10\% when nuclear is the only available firm generation technology.}

In general, the {\color{black}models developed and the} studies performed in \cite{sepulveda2021design, ziegler2019storage, albertus2020long, hargreaves2020long} evaluated the contribution of LDES to power systems while fixing different projections of technology costs in a fair attempt to answer the question: ``\textit{Given their projected costs, what is the value that LDES technologies can bring to a future system?}'', {\color{black}which assumes the cost of LDES is known}. Despite all the reported relevant findings, however, they do not provide a systematic methodology where the outcome directly informs {\color{black}boundary technology costs} of LDES for economic viability. 
{\color{black}Such systematic methodology would be particularly important for technologies that are still undergoing a major maturation process. Typically, the future (potentially competitive) cost of an non-mature technology is not yet defined. For instance, technology development policies and programs can bring dramatically down the costs of non-mature technologies in the course of a decade \cite{DOE_LDES_shot}.} Therefore, for non-mature technologies such as LDES, the valuation question is the opposite: ``\textit{How much does a technology need to cost to become economically viable to support a given policy target (e.g. decarbonization)?}''. In this work, we propose an approach to address the aforementioned question. Our framework is general enough to consider different policy targets and can be applied for the valuation of any non-mature technology to be included in power systems. Nonetheless, here we focus on computing the LDES boundary technology costs to support decarbonization.

\vspace{-0.2cm}
\subsection{Contributions}

In this paper, we propose a valuation methodology that estimates the boundary technology costs of LDES technologies for economic viability. We define this boundary as the technology cost below which the overall system costs (investment plus operations) will not exceed a reference value obtained when firm conventional (and already economically viable) generators perform the needed long-duration services instead of LDES. Our methodology combines the solution of two optimization problems and we use it to perform a realistic case study for the California's power system with high temporal resolution (8760h). Essentially, we estimate  the boundary technology costs of LDES technologies for economic viability in California for 2050 considering a reference energy matrix developed by NREL's Cambium 
\cite{gagnon2023cambium}.





The main contributions of this paper can be summarized as follows. 

\begin{enumerate}
    \item To propose a novel valuation methodology that computes the boundary technology cost of LDES systems {\color{black}below which these technologies become economically viable} {\color{black}based on the opportunity value maximization.} 
    \item To estimate the ideal technology costs to make 100-h LDES systems economically viable {\color{black}to support California's power system decarbonization goals.} 
    \item To improve the understanding of the operation of LDES systems as means of achieving a GHG emissions-free future.
\end{enumerate}

The remainder of this paper is organized as follows. Section \ref{sec:MathematicalFormulation} describes the mathematical formulations and our proposed framework. Section \ref{sec:CaseStudy} then delves into the case study conducted for the California's system considering the year of 2050.
This section discusses relevant financial and operational aspects related to LDES systems as candidate investment options. Finally, we draw our conclusions in Section \ref{sec:Conclusions}.


\vspace{-0.2cm}
\section{ Mathematical Formulation } \label{sec:MathematicalFormulation}

The main objective of the methodology proposed in this paper is to determine the boundary cost 
of LDES, which is the cost below which LDES becomes economically viable. Our methodology consists of a {\it baseline model} and an {\it opportunity value maximization model} which are solved in sequence as depicted in Fig. \ref{Fig:Methodology} to assess the maximum reduction in overall costs provided by a given quantity (energy and power) of LDES. We define this maximum reduction in overall costs as the maximum {\it opportunity value} of LDES, which has a direct relationship with its boundary cost. In the next subsections, we describe each of the aforementioned models in detail.




\begin{figure}[!tb]
    \centering
     \includegraphics[width=.45\textwidth,height=0.5\textheight,keepaspectratio]{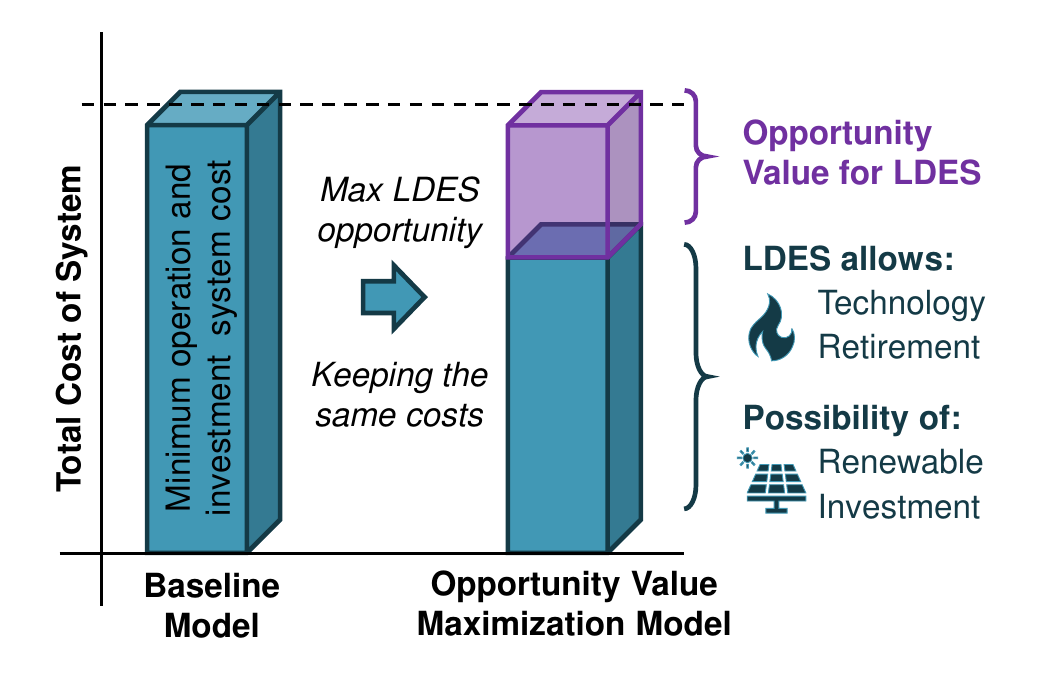}
    \caption{Methodology.}
    \vspace{-0.6cm}
    \label{Fig:Methodology}
\end{figure}

\vspace{-0.2cm}
\subsection{ Baseline model } \label{Baseline model}

As depicted in Fig. \ref{Fig:Methodology}, the {\it baseline model} provides a solution that minimizes the overall system costs. 
This optimization problem is solved without taking into account any policy related to enforcing the retirement generators of a particular technology. The main outcome of this model is the system overall cost which will be used as a benchmark to maximize the opportunity value of LDES. We describe the formulation of the {\it baseline model} as follows. 

\subsubsection{Objective function}

\begin{align}
    & q^* = \hspace{-5pt} \underset{{\substack{ \Delta^-_t, \Delta^+_t, \delta^{up,short}_t, p_{gt},\\ p^{st,ch}_{ht}, p^{st,dis}_{ht}, \overline{p}^{rem}_{g},\\ r^{st,up}_{ht}, r^{up}_{gt}, v_{ht},\\ x^{inv,gen \dagger}_{g}, x^{ret,gen \dagger}_{g},\\ x^{st,energy \dagger}_{g}, x^{st,power \dagger}_{g} 
    }}}{\text{Minimize}}  \hspace{0.1cm} \sum_{g\in G^{cand}} C^{inv,gen}_{g} x^{inv,gen \dagger}_{g} \notag \\
    & \hspace{5pt} + \sum_{h \in H^{short,cand}} \Bigl [ C^{st,energy}_{h} x^{st,energy \dagger}_{h} + C^{st,power}_{h} x^{st,power \dagger}_{h} \Bigr ]  \notag \\
    & \hspace{5pt} + \sum_{t \in T} \Biggl [ \sum_{g\in G} \Bigl [ C^p_{gt}p_{gt} + C^{up}_{gt} r^{up}_{gt} \Bigr] \notag \\
    & \hspace{5pt} + C^I \Bigl (\Delta^-_{t} + \Delta^+_{t} \Bigl ) + C^{short} \delta^{up,short}_{t}  \Biggr] \notag\\
    & + \sum_{g \in G^{firm,fixed}} C^{fom,gen}_{g} \overline{p}^{rem}_{g}  +  \sum_{g \in G^{renew,fixed}} C^{fom,gen}_{g} \overline{P}_{g} \notag\\
    & \hspace{5pt} + \sum_{g \in G^{cand}} C^{fom,gen}_{g} {x}^{inv,gen \dagger}_{g} \notag\\
    & \hspace{5pt} + \sum_{h \in H^{fixed}} C^{fom,st,power}_{h} \overline{P}^{st,power}_h \notag\\
    & \hspace{5pt} + \sum_{h \in H^{cand}} C^{fom,st,power}_{h} \Bigl ( \overline{P}^{st,power,ini}_h + x^{st,power \dagger}_{h}  \Bigr )  
    \label{BaselineModel_v1_1}
\end{align}

The objective function of the {\it baseline model} \eqref{BaselineModel_v1_1} minimizes expenditures on investments and operations as well as 
FO\&M costs. The investment expenditures are associated with new generation and short-duration storage systems, whereas operations expenditures are related to generation output, reserve provision, power imbalance, and reserve shortage. The FO\&M costs essentially apply for all existing and newly included generation and storage capacity.

\subsubsection{Balance}
    
\begin{align}
    & \sum_{g\in G} p_{gt} = D_{t} + \sum_{h \in H} \bigl [ p^{st,ch}_{ht} - p^{st,dis}_{ht} \bigr ] - \Delta^-_{t} + \Delta^+_{t};  \notag \\
    & \hspace{200pt} \forall t \in T \label{BaselineModel_v1_2}  \\
    & \Delta^-_{t}, \Delta^+_{t} \geq 0; \forall t \in T \label{BaselineModel_v1_3}
\end{align}

Power balance is enforced for each period $t$ via constraints \eqref{BaselineModel_v1_2}. In addition, the non-negative nature of the imbalance variables is imposed by constraints \eqref{BaselineModel_v1_3}.

\subsubsection{Reserve margin}
        
\begin{align}
    & \sum_{g \in G^{res,providers}} r^{up}_{gt} + \sum_{h \in H} r^{st,up}_{ht} \geq r^{up,min}_{t} \notag\\
    & \hspace{150pt} - \delta^{up,short}_{t}; \forall t \in T \label{BaselineModel_v1_4}  \\
    & \delta^{up,short}_{t} \geq 0; \forall t \in T \label{BaselineModel_v1_5}
\end{align}

Constraints \eqref{BaselineModel_v1_4} establish a minimum amount of reserve to be provisioned during each period $t$. Moreover, constraints \eqref{BaselineModel_v1_5} enforce non-negativity for reserve shortage variables $\delta^{up,short}_{t}$, which are penalized in the objective function. 

\subsubsection{Storage devices}
    
\begin{align}
    & v_{ht} = v_{h,t^{ini}} + \eta_h p^{st,ch}_{ht} - p^{st,dis}_{ht}; \forall h \in H, t = 1 \label{BaselineModel_v1_6} \\ 
    & v_{ht} = v_{h,t-1} + \eta_h p^{st,ch}_{ht} - p^{st,dis}_{ht}; \forall h \in H, t \in T | t \geq 2 \label{BaselineModel_v1_7} \\ 
    & v_{h,t^{ini}} = v_{h,t^{end}} 
    ; \forall h \in H \label{BaselineModel_v1_8} \\
    & \underline{V}_h \leq  v_{ht} \leq \overline{V}_h; \forall h \in H^{fixed}, t \in T \label{BaselineModel_v1_9} \\
    & \underline{V}_h \leq  v_{ht} \leq \bigl ( \overline{V}^{ini}_h + x^{st,energy \dagger}_{h} \bigr ); \forall h \in H^{cand}, t \in T \label{BaselineModel_v1_10} \\
    & 0 \leq p^{st,ch}_{ht} \leq \overline{P}^{st,power}_h; \forall h \in H^{fixed}, t \in T \label{BaselineModel_v1_11} \\
    & 0 \leq p^{st,dis}_{ht} + r^{st,up}_{ht} \leq \overline{P}^{st,power}_h; \forall h \in H^{fixed}, \notag\\
    & \hspace{200pt} t \in T \label{BaselineModel_v1_12} \\
    & 0 \leq p^{st,ch}_{ht} \leq \overline{P}^{st,power,ini}_h +  x^{st,power \dagger}_{h}; \notag\\
    & \hspace{145pt} \forall h \in H^{cand}, t \in T \label{BaselineModel_v1_13} \\
    & 0 \leq p^{st,dis}_{ht} + r^{st,up}_{ht} \leq \overline{P}^{st,power,ini}_h + x^{st,power \dagger}_{h}; \notag\\
    & \hspace{145pt} \forall h \in H^{cand}, t \in T \label{BaselineModel_v1_14} \\
    & v_{ht} - r^{st,up}_{ht} \geq \underline{V}_h; \forall h \in H, t \in T \label{BaselineModel_v1_15} \\
    & x^{st,energy \dagger}_{h} = x^{st,power \dagger}_{h} \overline{S}_h; \forall h \in H^{cand} \label{BaselineModel_v1_16} \\
    & 0 \leq x^{st,power \dagger}_{h} \leq \overline{x}^{st,power}_{h}; \forall h \in H^{cand} \label{BaselineModel_v1_17} \\
    & 0 \leq x^{st,energy \dagger}_{h} \leq \overline{x}^{st,energy}_{h}; \forall h \in H^{cand} \label{BaselineModel_v1_18} \\
    & p^{st,ch}_{ht}, p^{st,dis}_{ht}, 
    r^{st,up}_{ht} \geq 0; \forall h \in H, t \in T \label{BaselineModel_v1_19} 
\end{align}

Constraints \eqref{BaselineModel_v1_6} and \eqref{BaselineModel_v1_7} update the state of charge of each storage system $h$ throughout the time periods. Constraints \eqref{BaselineModel_v1_8} impose equality between initial and final states of charge of each storage system. Limits for states of charge are enforced via constraints \eqref{BaselineModel_v1_9} and \eqref{BaselineModel_v1_10} for existing and candidate storage systems. Analogously, constraints \eqref{BaselineModel_v1_11}--\eqref{BaselineModel_v1_14} express power charging and discharging capacities for existing and candidate storage systems. Constraints \eqref{BaselineModel_v1_15} ensure that there is sufficient state of charge to hold the provisioned reserves. Constraints \eqref{BaselineModel_v1_16} impose a relationship between power and energy capacity of each candidate storage system $h$ according to the specified  duration $\overline{S}_h$. Limits for investments in candidate storage systems are enforced by \eqref{BaselineModel_v1_17} and \eqref{BaselineModel_v1_18} and the non-negativity of power charging and discharging as well as reserve provision variables is enforced in \eqref{BaselineModel_v1_19}.    

\subsubsection{Generators}

\begin{align}
    & 0 
    \leq p_{gt} \leq \overline{p}^{rem}_{g} - r^{up}_{gt}; \forall g \in G^{firm,fixed}, t \in T \label{BaselineModel_v1_20} \\
    & 0 \leq p_{gt} \leq \overline{P}_g f^{available}_{gt}; \notag\\
    & \hspace{38pt} \forall g \in G^{renew,fixed}~|~g \notin G^{res,providers} 
    , t \in T \label{BaselineModel_v1_21} \\
    & 0 
    \leq p_{gt} \leq \overline{P}_g f^{available}_{gt} - r^{up}_{gt}; \notag\\
    & \hspace{46pt}\forall g \in G^{renew,fixed} ~\cap~ G^{res,providers} 
    , t \in T \label{BaselineModel_v1_22} \\    
    & 0 \leq r^{up}_{gt} \leq \overline{p}^{rem}_{g} \overline{R}^{up,factor}_g ; \forall g \in G^{firm,fixed}, t \in T \label{BaselineModel_v1_23} \\
    & 0 \leq r^{up}_{gt} \leq \overline{P}_g f^{available}_{gt} \overline{R}^{up,factor}_g  ; \notag\\
    & \hspace{46pt} \forall g \in G^{renew,fixed} ~\cap~ G^{res,providers} 
    , t \in T  \label{BaselineModel_v1_24} \\
    & 0 
    \leq p_{gt} \leq {x}^{inv,gen \dagger}_{gt} - r^{up}_{gt}; \forall g \in G^{firm,cand}, t \in T \label{BaselineModel_v1_25} \\
    & 0 \leq p_{gt} \leq x^{inv,gen \dagger}_{g} f^{available}_{gt};\notag\\
    & \hspace{41pt}  \forall g \in G^{renew,cand}~|~g \notin G^{res,providers}
    , t \in T \label{BaselineModel_v1_26} \\
    & 0 
    \leq p_{gt} \leq x^{inv,gen \dagger}_{g} f^{available}_{gt} - r^{up}_{gt};\notag\\
    & \hspace{50pt} \forall g \in G^{renew,cand}  ~\cap~ G^{res,providers}
    , t \in T \label{BaselineModel_v1_27} \\ 
    & 0 \leq r^{up}_{gt} \leq {x}^{inv,gen \dagger}_{gt} \overline{R}^{up,factor}_g;\notag\\
    & \hspace{128pt} \forall g \in G^{firm,cand}, t \in T \label{BaselineModel_v1_28} \\
    & 0 \leq r^{up}_{gt} \leq {x}^{inv,gen \dagger}_{g} f^{available}_{gt} \overline{R}^{up,factor}_g; \notag\\
    & \hspace{48pt} \forall g \in G^{renew,cand} ~\cap~  G^{res,providers}, t \in T \label{BaselineModel_v1_29} \\
    & 0 \leq x^{inv,gen \dagger}_{g} \leq \overline{x}^{inv,gen}_{g} ; \forall g \in G^{cand} \label{BaselineModel_v1_30} \\
    & p_{gt} - p_{g,t-1} \leq RU^{factor}_g \overline{p}^{rem}_{g};\notag\\ 
    & \hspace{100pt} \forall g \in G^{firm,fixed}, t \in T | t \geq 2 \label{BaselineModel_v1_31} \\
    & p_{g,t-1} - p_{gt} \leq RD^{factor}_g \overline{p}^{rem}_{g};\notag\\ 
    & \hspace{100pt} \forall g \in G^{firm,fixed}, t \in T | t \geq 2 \label{BaselineModel_v1_32} \\
    & p_{gt} - p_{g,t-1} \leq RU^{factor}_g {x}^{inv,gen \dagger}_{gt};\notag\\
    & \hspace{100pt} \forall g \in G^{firm,cand}, t \in T | t \geq 2 \label{BaselineModel_v1_33} \\
    & p_{g,t-1} - p_{gt} \leq RD^{factor}_g {x}^{inv,gen \dagger}_{gt}; \notag\\
    &\hspace{100pt} \forall g \in G^{firm,cand}, t \in T | t \geq 2 \label{BaselineModel_v1_34} \\
    & \overline{p}^{rem}_{g} = \overline{P}_{g} - x^{ret,gen \dagger}_{g}; \forall g \in G^{firm,fixed} \label{BaselineModel_v1_35} \\
    &  \underline{x}^{ret,gen}_{g} \overline{P}_g\leq x^{ret,gen \dagger}_{g} \leq \overline{x}^{ret,gen}_{g} \overline{P}_g ; \notag \\
    &\hspace{150pt}  \forall g \in G^{firm,fixed} \label{BaselineModel_v1_36} 
\end{align}

Constraints \eqref{BaselineModel_v1_20} impose generation limits for existing firm units, whereas \eqref{BaselineModel_v1_21} and \eqref{BaselineModel_v1_22} do the same for existing renewable units. The maximum reserve provision capacities of existing firm and renewable units are expressed through \eqref{BaselineModel_v1_23} and \eqref{BaselineModel_v1_24}. Likewise, candidate units have limits enforced on their generation output and reserve provision via constraints \eqref{BaselineModel_v1_25}--\eqref{BaselineModel_v1_29}. Constraints \eqref{BaselineModel_v1_30} establish a maximum amount of generation capacity to be installed for each candidate unit. Generation ramping is modeled via constraints \eqref{BaselineModel_v1_31}--\eqref{BaselineModel_v1_34}. Finally, constraints \eqref{BaselineModel_v1_35} indicate the remaining capacity of each generator $g$ after total or partial retirement and constraints \eqref{BaselineModel_v1_36} limit capacity retirement.

\vspace{-0.4cm}
\subsection{Opportunity value maximization model}

The {\it opportunity value maximization model}, formulated as \eqref{OpportunityValueMaximization_v1_1}--\eqref{OpportunityValueMaximization_v1_3}, seeks to obtain the maximum technology cost of LDES below which the overall costs associated with investments, operations, and maintenance are lower or equal to $q^*$ while complying with all technical constraints considered in the {\it baseline model}. The cost bound $q^*$ is the value of the objective function \eqref{BaselineModel_v1_1} and imposing this bound essentially ensures that the resulting opportunity value of LDES will not imply an overall system cost higher than that attained through the {\it baseline model}. 
\begin{align}
    & 
    \underset{{\substack{ \Delta^-_t, \Delta^+_t, \delta^{up,short}_t, c^{BC},\\ p_{gt}, p^{st,ch}_{ht}, p^{st,dis}_{ht}, \overline{p}^{rem}_{g},q^{over}\\ r^{st,up}_{ht}, r^{up}_{gt}, v_{ht},x^{inv,gen \dagger}_{g},\\  x^{ret,gen \dagger}_{g}, x^{st,energy \dagger}_{g}, x^{st,power \dagger}_{g} 
    }}}{\text{Maximize}}  \hspace{0.1cm} c^{BC}  - C^{over} q^{over} \label{OpportunityValueMaximization_v1_1}\\
    &\text{Subject to:}\notag \\
    & \sum_{h\in H^{long,cand} }  c^{BC} x^{st,power \ddagger}_{h}  \notag \\ 
    & \hspace{8pt} + \sum_{g\in G^{cand}}  C^{inv,gen}_{ge} x^{inv,gen \dagger}_{g}  \notag \\
    & \hspace{8pt}  + \sum_{h \in H^{short,cand}} \Bigl [ C^{st,energy}_{h} x^{st,energy \dagger}_{h} + C^{st,power}_{h} x^{st,power \dagger}_{h} \Bigr ] \notag \\
    \notag \\
    & \hspace{8pt} +  \sum_{t \in T} \Biggl [ \sum_{g\in G} \Bigl [ C^p_{gt}p_{gt} + C^{up}_{gt} r^{up}_{gt} \Bigr] \notag \\
    & \hspace{8pt} + C^I \Bigl (\Delta^-_{t} + \Delta^+_{t} \Bigl ) + C^{short} \delta^{up,short}_{t} \Biggr] \notag\\
    & \hspace{8pt}+  \sum_{g \in G^{firm,fixed}} C^{fom,gen}_{g} \overline{p}^{rem}_{g} +  \sum_{g \in G^{renew,fixed}} C^{fom,gen}_{g} \overline{P}_{g} \notag\\
    & \hspace{8pt} + \sum_{g \in G^{cand}} C^{fom,gen}_{g} {x}^{inv,gen \dagger}_{g} \notag\\
    & \hspace{8pt} + \sum_{h \in H^{fixed}} C^{fom,st,power}_{h} \overline{P}^{st,power}_h   \notag\\
    & \hspace{8pt} + \sum_{h \in H^{cand}} C^{fom,st,power}_{h} \Bigl ( \overline{P}^{st,power,ini}_h \notag\\
    & \hspace{8pt}  + x^{st,power \dagger}_{h} \Bigr ) 
    \leq q^* + q^{over}\label{OpportunityValueMaximization_v1_2} \\
    & \text{Constraints \eqref{BaselineModel_v1_2}-\eqref{BaselineModel_v1_36}}  \label{OpportunityValueMaximization_v1_3}
\end{align}

For a given user-defined amount of newly installed LDES power capacity $\sum_{h\in H^{long,cand} } x^{st,power \ddagger}_{h}$, model \eqref{OpportunityValueMaximization_v1_1}--\eqref{OpportunityValueMaximization_v1_3} maximizes the opportunity value of LDES while penalizing the auxiliary variable $q^{over}$, which relaxes constraint \eqref{OpportunityValueMaximization_v1_2} in case it is not possible to comply with constraints \eqref{OpportunityValueMaximization_v1_3} while respecting the cost bound $q^*$. We define the term $\sum_{h\in H^{long,cand} }  c^{BC} x^{st,power \ddagger}_{h}$ in \eqref{OpportunityValueMaximization_v1_2} as the opportunity value of LDES. Furthermore, according to Proposition \ref{prop_1}, the optimal solution for model \eqref{OpportunityValueMaximization_v1_1}--\eqref{OpportunityValueMaximization_v1_3} provides $c^{BC^*}$, which is the boundary cost of LDES.


\begin{prop}\label{prop_1}
Considering a predefined quantity of newly installed LDES, $\sum_{h\in H^{long,cand} } x^{st,power \ddagger}_{h}$, and supposing that \eqref{OpportunityValueMaximization_v1_1}--\eqref{OpportunityValueMaximization_v1_3} is feasible for $q^{over}=0$, in the optimal solution for \eqref{OpportunityValueMaximization_v1_1}--\eqref{OpportunityValueMaximization_v1_3}, $c^{BC^*}$ will be a boundary cost in \$/MW below which the inclusion of investment costs related to LDES will not imply an overall system cost higher than $q^*$.
\end{prop}

\begin{proof}
Note that, according to the structure of  \eqref{OpportunityValueMaximization_v1_1} and \eqref{OpportunityValueMaximization_v1_2}, investment, operational, and maintenance costs will be minimized in the {\it opportunity value maximization model} as they are in the {\it baseline model}. This cost minimization will occur in order to maximize the variable $c^{BC}$ as much as possible for a given predetermined newly installed capacity of LDES, represented by $x^{st,power \ddagger}_{h}$ for each storage $h$. 
In this context, if $q^{over}=0$ is a feasible solution and $C^{over}$ is high enough, in the optimal solution, the term $\sum_{h\in H^{long,cand} }  c^{BC^*} x^{st,power \ddagger}_{h}$ in \eqref{OpportunityValueMaximization_v1_2} will be equivalent to the maximum reduction in the overall cost of the system (compared to $q^*$) as a result of the installation of a certain amount, $\sum_{h\in H^{long,cand} } x^{st,power \ddagger}_{h}$, of LDES in the system. Therefore, (i) the term $\sum_{h\in H^{long,cand} }  c^{BC^*} x^{st,power \ddagger}_{h}$ will be the corresponding maximum opportunity value of LDES and, consequently, (ii) $c^{BC^*}$ will be a threshold or boundary cost in \$/MW below which the integration of a predefined quantity of LDES, $\sum_{h\in H^{long,cand} } x^{st,power \ddagger}_{h}$, will reduce the investment (not counting LDES), operational, and maintenance costs such that the inclusion of investment costs related to LDES will not imply an overall system cost higher than $q^*$. 
\end{proof}

\vspace{-0.4cm}
\subsection{Solution framework}

As previously mentioned, the {\it baseline model} and the {\it opportunity value maximization model} are solved in sequence. Within our proposed framework, different circumstances can be considered to evaluate the boundary cost of LDES. In this work, we focus specifically on determining the opportunity value of LDES to achieve a fully decarbonized power system which entails a complete retirement of gas generators and a potential increase in renewable installed capacity. To do so, we obtain $q^*$ by solving the {\it baseline model} while enforcing $\overline{x}^{st,energy}_{h} = 0, \forall h \in H^{cand} $ in constraints \eqref{BaselineModel_v1_18}; $\overline{x}^{inv,gen}_{g}=0 , \forall g \in G^{cand}  $ in constraints \eqref{BaselineModel_v1_30}; and $\underline{x}^{ret,gen}_{g} = \overline{x}^{ret,gen}_{g} = 0, \forall g \in G^{firm,fixed}$ in constraints \eqref{BaselineModel_v1_36}, which essentially means that no investments and no retirements are allowed in the {\it baseline model}. Furthermore, we write the {\it opportunity value maximization model} with $x^{st,energy \dagger}_{h} = x^{st,energy \ddagger}_{h}, \forall h \in H^{long,cand} $  (where $x^{st,energy \ddagger}_{h}$ is predefined newly installed capacity for storage system $h$) and $\overline{x}^{st,energy}_{h} \geq 0, \forall h \in H^{short,cand} $ in constraints \eqref{BaselineModel_v1_18}; $\overline{x}^{inv,gen}_{g} \geq 0, \forall g \in G^{renew,cand}  $ and $\overline{x}^{inv,gen}_{g}=0, \forall g \in G^{firm,cand} $ in constraints \eqref{BaselineModel_v1_30}; and $\underline{x}^{ret,gen}_{g} = \overline{x}^{ret,gen}_{g} = 1, \forall g \in G^{gas,fixed}$ and $\underline{x}^{ret,gen}_{g} = \overline{x}^{ret,gen}_{g} = 0, \forall g \in G^{firm,fixed} \setminus G^{gas,fixed}$ in constraints \eqref{BaselineModel_v1_36}.

\vspace{-0.6cm}
\section{Case Study}\label{sec:CaseStudy}
\vspace{-0.1cm}

In this case study, we focus on the the critical role of LDES technologies in providing the necessary means to enable the California's power system to move toward full decarbonization in 2050. Gas power plants are currently the main source of firm generation to counterbalance renewable intermittence in the state. However, in a fully decarbonized future power system, gas units should be replaced by a more environmentally friendly option. This firm generation role can be performed by LDES technologies. Since these technologies are still not mature, here we estimate, via the proposed methodology, the boundary costs that would make them economically viable in for the California's system.

The numerical experiments have been implemented in Python and performed on a machine with one Intel\textregistered ~Corel\texttrademark ~i9-12900K 3.20 GHz and 32 GB of RAM, using Gurobi 10.0.1 under Pyomo. {\color{black}For replicability purposes, the input data can be downloaded
from \cite{dataset_california_system}.}

In the next subsections, we present an overview of the data source that we consider in this case study, a discussion about the projected California's energy matrix for 2050 without any LDES technology as well as the financial and operational results of our case study.

\begin{figure}[!tb]
    \centering
    \includegraphics[width=.35\textwidth,height=0.30
     \textheight,keepaspectratio]{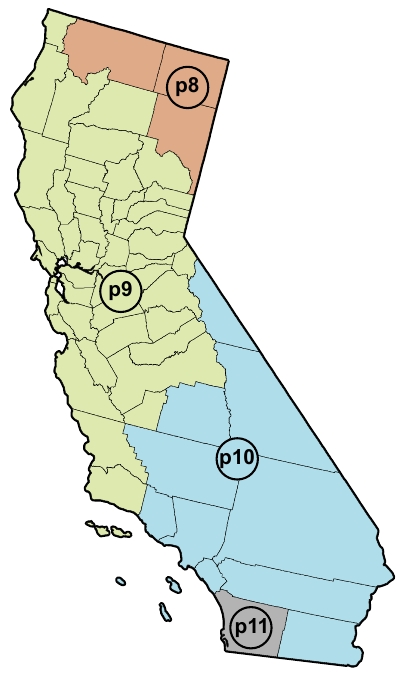}
    \caption{Map of California with BA representatives covered by the model. Adapted from \cite{gagnon2023cambium}.}
    \vspace{-0.4cm}
    \label{Fig:California_map}
\end{figure}

\subsection{Data source}

The dataset used in our computational experiments is based on a compilation of data from several reliable sources. These sources provide a comprehensive view of how the California's energy matrix is projected to be in 2050. We accessed information about system load as well as available generation technologies with their respective installed capacities through data made available by NREL's Cambium 2022 \cite{gagnon2023cambium_site}. More specific parameters such as ramp rates of each power plant technology come from the Cambium 2022 documentation \cite{gagnon2023cambium}. Furthermore, we also consider capacity factors from the Cambium 2022 \cite{gagnon2023cambium_site}, which are essentially computed as the ratio between effective provided generation at each time period under consideration and installed capacity.

Values of fixed operation and maintenance (FO\&M) costs of each generator are provided by the Regional Energy Deployment System (ReEDS)  base \cite{2021reeds}, which is publicly available on GitHub \cite{reedsGitHub}, and complemented by information present in the 2022 Annual Technology Baseline (ATB) \cite{atb2022}. In addition, we use fuel prices that are reported in the AEO 2023 \cite{AEO_2023}. Moreover, from the ATB 2022, we take into account relevant data related to investment costs associated with renewable power plants.

Our financial results in this study are expressed in 2022 U.S. dollars. Monetary values are updated considering the Consumer Price Index (CPI) \cite{usinflationcalculator}.

\begin{table}
\caption{Installed capacity and the limit of investment by technology with the respective number of generators.}
\begin{threeparttable}[b]
\footnotesize
\begin{tabularx}{\linewidth}{lrrrr}
\hline
\textbf{Technology} & \multicolumn{1}{c}{\textbf{\begin{tabular}[c]{@{}c@{}}Installed\\ Capacity \\ (GW)\end{tabular}}} & \multicolumn{1}{c}{\textbf{\begin{tabular}[c]{@{}c@{}}Numb. of \\ Existing \\ Gen.\end{tabular}}} & \multicolumn{1}{c}{\textbf{\begin{tabular}[c]{@{}c@{}}Limit of \\ Invest. \\ (GW)\end{tabular}}} & \multicolumn{1}{c}{\textbf{\begin{tabular}[c]{@{}c@{}}Numb. of \\ Candidate \\ Gen.\end{tabular}}} \\ \hline
Biopower & 0.29 & 8 & - & - \\
Distributed PV & 28.50 & 4 & 28.50 & 4 \\
Distributed utility PV & 0.04 & 3 & 0.04 & 3 \\
Natural gas cc & 18.39 & 9 & - & - \\
Natural gas ct & 8.99 & 10 & - & - \\
Geothermal & 2.23 & 3 & - & - \\
Hydropower ED\tnote{a} & 3.49 & 2 & - & - \\
Hydropower END\tnote{b} & 6.70 & 4 & - & - \\
Hydropower UD\tnote{c} & 1.09 & 1 & - & - \\
Hydropower UND\tnote{d} & 0.15 & 3 & - & - \\
Landfill gas & 0.07 & 9 & - & - \\
Oil/gas steam & 0.15 & 9 & - & - \\
Utility-scale PV & 83.59 & 4 & 83.59 & 4 \\
Offshore Wind & 25.00 & 2 & 25.00 & 2 \\
Onshore Wind & 9.48 & 4 & 9.48 & 4 \\ \hline
\textbf{Total} & \textbf{188.16} & \textbf{75} & \textbf{146.61} & \textbf{17} \\ \hline
\end{tabularx}
\begin{tablenotes}
\item [a] Existing dispatchable.
\item [b] Existing non-dispatchable.
\item [c] Undiscovered dispatchable.
\item [d] Undiscovered non-dispatchable.
\vspace{-0.8cm}
\end{tablenotes}
\label{generators_capacity}
\end{threeparttable}
\end{table}

\begin{table}
\caption{Power capacity and the limit of investment by technology with the respective number of storage.}
\begin{threeparttable}[b]
\footnotesize
\begin{tabular}{lrrrr}
\hline
\textbf{Technology} & \multicolumn{1}{c}{\textbf{\begin{tabular}[c]{@{}c@{}}Power \\ Capacity \\ (GW)\end{tabular}}} & \multicolumn{1}{c}{\textbf{\begin{tabular}[c]{@{}c@{}}Numb. of \\ Existing \\ Storage\end{tabular}}} & \multicolumn{1}{c}{\textbf{\begin{tabular}[c]{@{}c@{}}Limit of \\ Invest. \\ (GW)\end{tabular}}} & \multicolumn{1}{c}{\textbf{\begin{tabular}[c]{@{}c@{}}Numb. of \\ Candidate \\ Storage\end{tabular}}} \\ \hline
Battery   2h & 0.30 & 2 & - & - \\
Battery 4h & 17.46 & 4 & 43.00 & 1 \\
Battery 6h & 9.55 & 3 & - & - \\
Battery 8h & 4.03 & 4 & - & - \\
PHS   (12h) & 11.43 & 3 & - & - \\
LDES & - & - & 0 - 75\tnote{a} & 1 \\ \hline
\textbf{Total} & \textbf{42.78} & \textbf{16} & \multicolumn{1}{r}{\textbf{-}} & \textbf{2} \\ \hline
\end{tabular}
\begin{tablenotes}
\item [a] The range consists of discrete values between 0 and 75.
\end{tablenotes}
\label{storage_capacity}
\end{threeparttable}
\end{table}

\vspace{-0.2cm}
\subsection{California system} \label{Sec:California_System}

Following the system representation used by the databases of Cambium and ReEDS, in this paper, we consider four major balancing areas (BA) in the California's territory as depicted in Fig. \ref{Fig:California_map}. Our analysis for this four-area system comprises an hourly temporal resolution with 8760 intervals throughout the year. In addition, in accordance with current standards, we impose a reserve requirement equivalent to 15\% of the system demand for all time periods \cite{caiso_report_2023}.

The database provided by Cambium projects a system with 714 existing generators in California for the year of 2050. For our case study, we applied clustering using the K-means method, aiming for three representative generators of each technology (high, low, and medium cost) within each BA. The clustering criterion is the generation cost of the generator, i.e., the $C^{p}_{gt}$ used in the model, resulting in a database with 75 fixed existing generators.

In addition to the fixed generators, we also consider a group of candidate generators that are all are renewables. Essentially, for each already existing renewable generator in 2050 (according to the projected data) within each BA, we set up a candidate renewable generator with the same specifications apart from the installed capacity, which is determined by the solution of our methodology. The resulting 17 candidate generators are listed in Table \ref{generators_capacity}.

\begin{figure}[!tb]
    \centering
    \includegraphics[width=.45\textwidth,height=0.5
     \textheight,keepaspectratio]{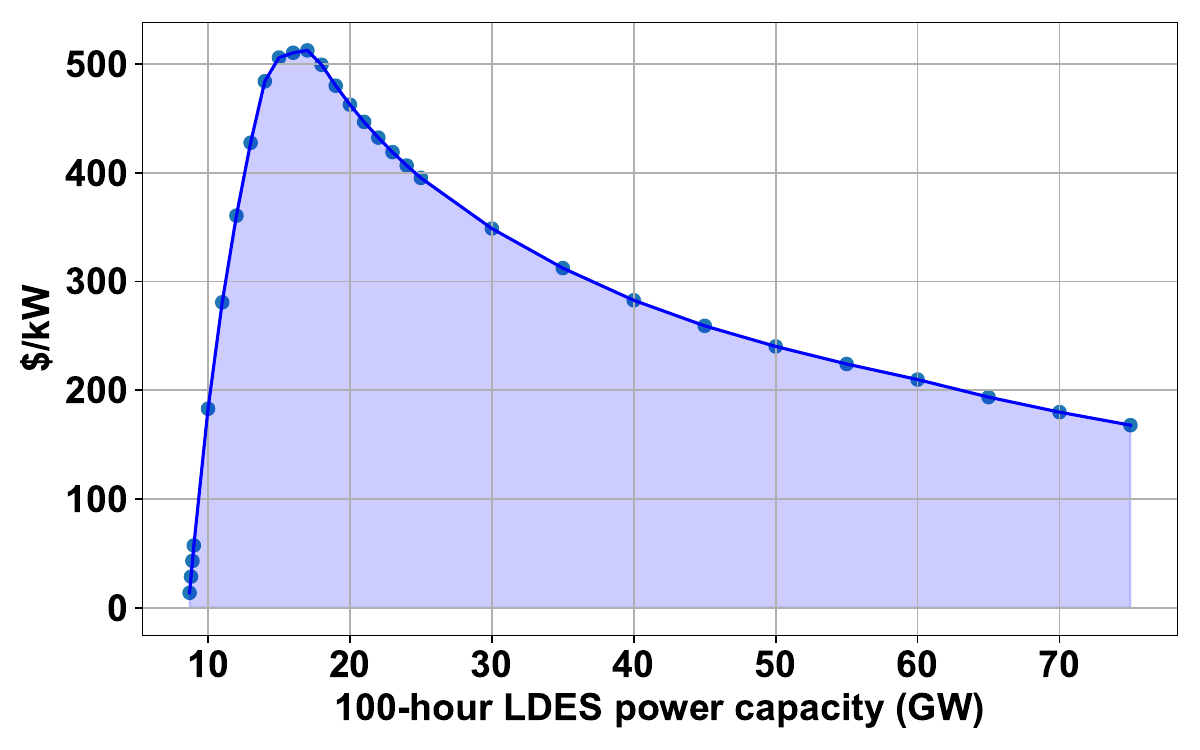}
    \caption{The boundary costs of LDES below which these technologies will be economically viable for the California's system in 2050. 
    }
    \label{Fig:Boundary_cost_power}
\end{figure}

Also, in this study, we include, as investment options, candidate SDES and LDES systems as summarized in Table \ref{storage_capacity}. The duration of the candidate SDES option is 4 hours and its maximum capacity investment is set to a value equivalent to the aggregate installed storage power capacity in California in 2050. This maximum capacity investment limit entails that our investment results are determined in terms of aggregate numbers for SDES systems and, therefore, may comprise a collection of SDES units to compose the indicated capacity to be added to the system. This aggregate assumption is also taken for the candidate LDES system which has a 100-hour duration and capacity limits ranging from 0 to 75 GW in a discrete manner. RTE values are set equal to 85\% and 42.5\% for SDES and LDES systems, respectively.

\begin{figure}[!tb]
    \centering
    \includegraphics[width=.45\textwidth,height=0.45
     \textheight,keepaspectratio]{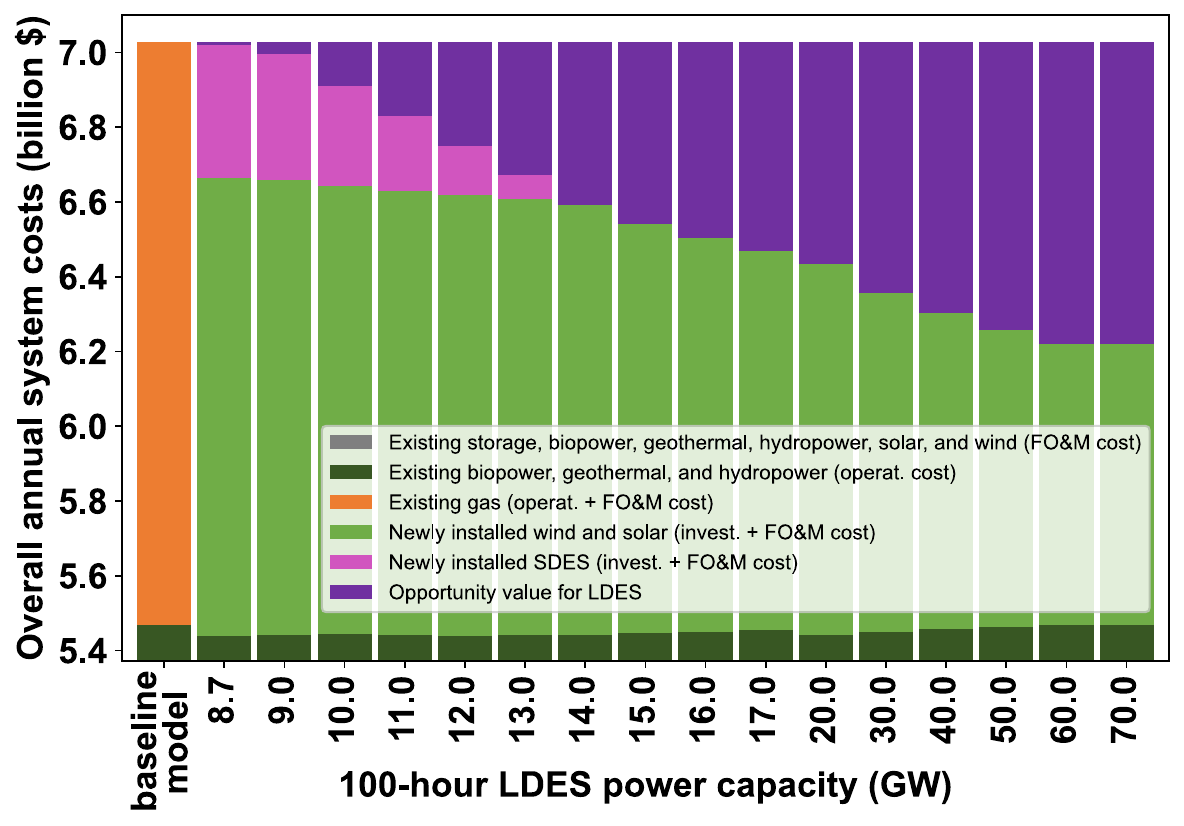}
    \caption{Opportunity values in purple associated with different power capacities of 100-hour LDES. The first bar corresponds to the system costs determined by the {\it baseline model} while considering generation from gas units and no inclusion of LDES systems. {\color{black}For illustration purposes, the FO\&M of some existing generators (gray bars which are ommitted since they are the same for all 100-hour LDES power capacities under consideration in this graph) and storages are omitted from the figure.}
    }
    %

    \label{Fig:Opportunity_value_LDES}
\end{figure}

\begin{figure}[!tb]
    \centering
    \includegraphics[width=.36\textwidth,height=0.35
     \textheight,keepaspectratio]
     {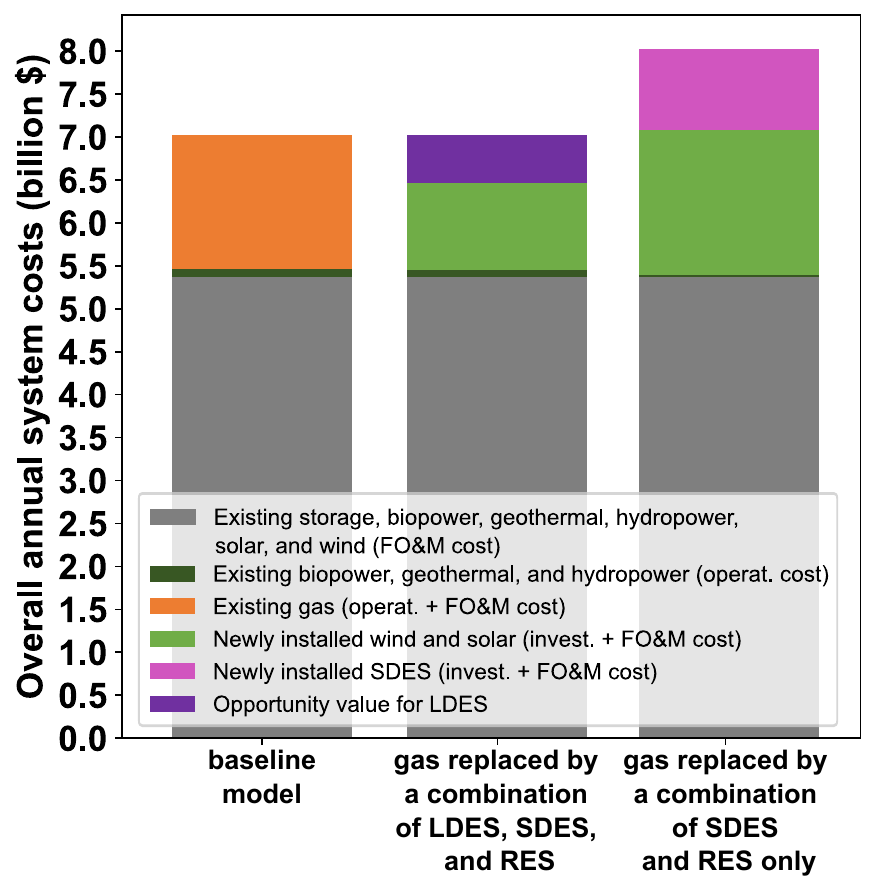}
    \caption{{\color{black}Comparison between (i) the reference costs from the {\it baseline model}, (ii) the resulting costs from the {\it opportunity value maximization model} when the gas is replaced by a combination of LDES, SDES, and RES, and (iii) the resulting costs from the {\it opportunity value maximization model} when the gas is replaced by a combination of SDES and RES only.} }
    \vspace{-0.2cm}
    \label{Fig:Cost_without_LDES}
\end{figure}

\begin{figure}[!tb]
    \centering
    \includegraphics[width=.45\textwidth,height=0.5
     \textheight,keepaspectratio]{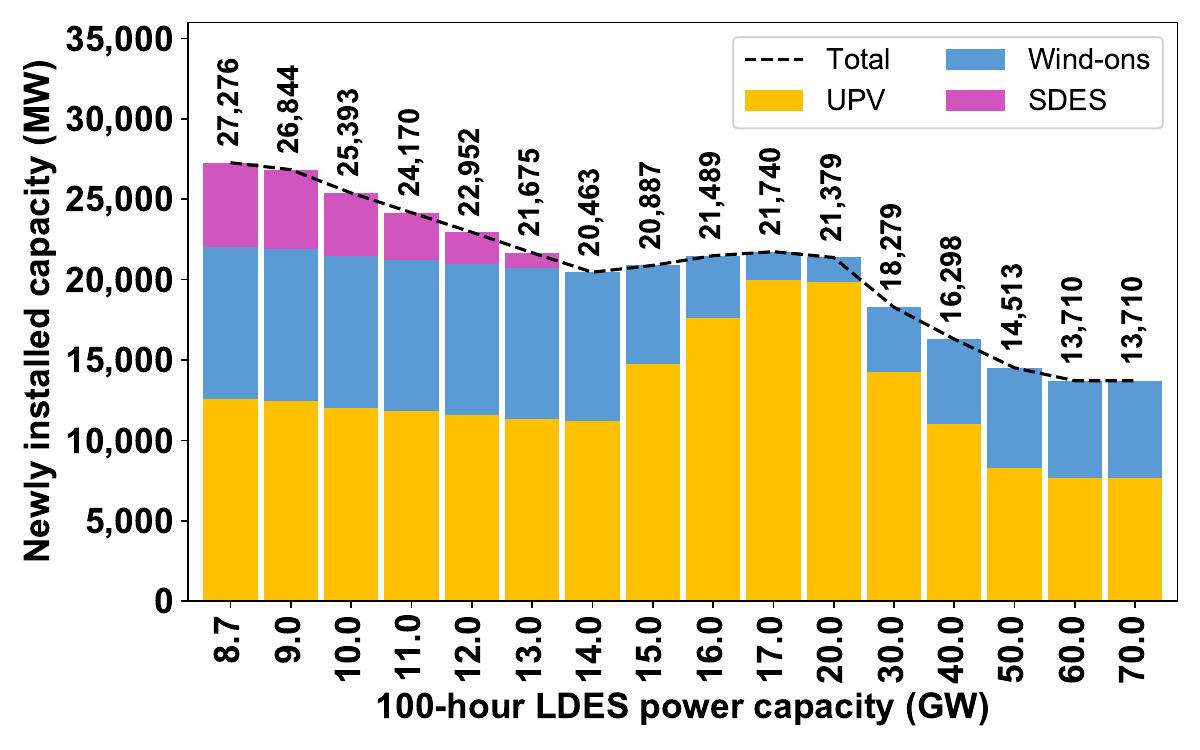}
    \caption{Additional investment required in California in 2050 to have a system without gas energy sources.}
    \vspace{-0.5cm}
    \label{Fig:Additional_investment}
\end{figure}

\subsection{{\color{black}Economic} results}

The boundary costs of LDES systems are defined in this paper as the {\color{black}technology costs} below which these technologies become economically viable. To obtain these boundary costs for California, we maximized the opportunity value following our proposed solution framework. As a reference, we consider the previously described data for the projected California's energy matrix in 2050 without any candidate assets to solve the {\it baseline model}. Then, we run the {\it opportunity value maximization model} for different amounts of LDES added to the system while bounding overall system costs below those of the {\it baseline model}, retiring 100\% of the gas power plants and considering candidate investments in SDES and renewable generators.

Our results indicate that at least {\color{black}8.7 GW} of 100-hour LDES will be necessary to retire gas power plans and maintain overall system costs limited to the same values determined by the {\it baseline model}, in which the participation of gas units is still included. According to Fig. \ref{Fig:Boundary_cost_power}, the boundary cost of {\color{black}8.7 GW} of 100-hour LDES is {\color{black}US\$13.74 $\text{kW}^{-1}$}. It is worth noting that, as we increase the power capacity of 100-hour LDES in Fig. \ref{Fig:Boundary_cost_power}, overall system costs decrease as a result of lower needed operational and investment costs to retire gas units (as it will be seen in the next graphs). Therefore, the boundary cost of LDES increases from  {\color{black}8.7 GW} to 17 GW, where it reaches a peak of US\$512.54 $\text{kW}^{-1}$. Then, after its peak, the boundary cost of LDES starts to fall since the rate of reduction in overall system costs per additional kW of LDES power capacity begins to decline.

The opportunity value for various amounts of LDES is shown in Fig. \ref{Fig:Opportunity_value_LDES}. The first bar on the left depicts the reference system costs obtained through the {\it baseline model}, including the operational and fixed costs of gas plants. {\color{black}The FO\&M costs of some existing generators (biopower, geothermal, hydropower, solar, and wind) and existing storages are the same for all bars, therefore they are omitted from Fig. \ref{Fig:Opportunity_value_LDES}.} It is worth mentioning that an eventual bar corresponding to a 100-hour LDES power capacity lower than 8.7 GW would result in an overall system cost higher than the reference cost obtained when gas units are still present in the system. In addition, it is interesting to see that, as the amount of 100-hour LDES power capacity increases, the need for SDES and renewable investments significantly decreases, which makes room for the opportunity value maximization of the LDES option.

{\color{black}In Fig. \ref{Fig:Cost_without_LDES}, we provide a direct comparison between (i) the reference costs obtained through the {\it baseline model}, (ii) the costs obtained via the {\it opportunity value maximization model} when gas is replaced by a combination of LDES (setting the 100-hour LDES power capacity to 17 GW), SDES and RES, and (iii) the costs obtained via the {\it opportunity value maximization model} when gas is replaced by a combination of SDES and RES only.} From these results, it is clear that the system would be substantially more expensive if we retire gas plants without LDES. On the other hand, when 17 GW of 100-hour LDES is present, overall system costs are the same as the reference cost as long as the cost associated with the LDES is at {\color{black}US\$512.54 $\text{kW}^{-1}$}.

\begin{figure}[!tb]
    \centering
    \includegraphics[width=.45\textwidth,height=0.45
     \textheight,keepaspectratio]{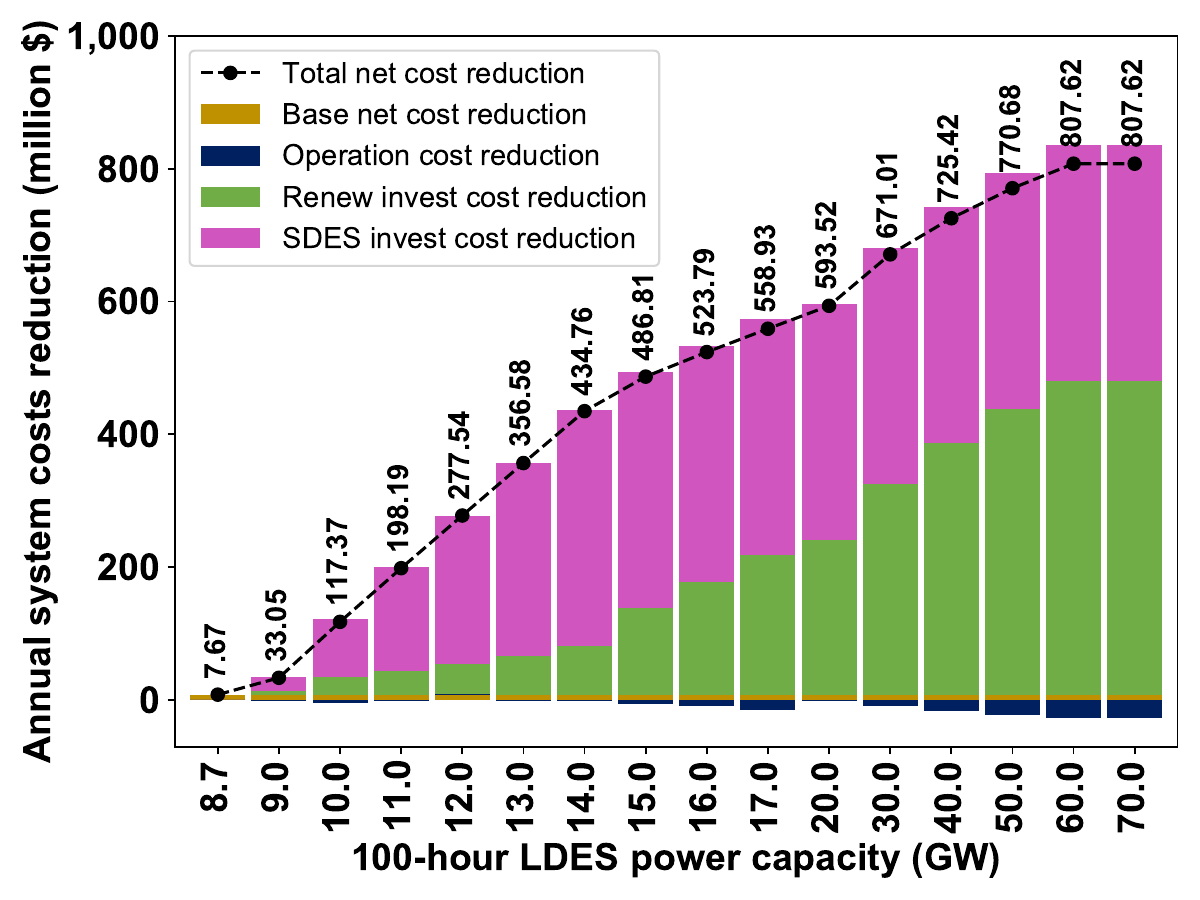}
    \caption{The cost reduction in the California system in 2050 for different quantities of LDES when the gas plants are retired. 
    The base net cost reduction refers to decrease in overall costs once 8.7GW of LDES are included in the system compared to when gas units are still present.}
    \vspace{-0.4cm}
    \label{Fig:Cost_Reduction}
\end{figure}

\begin{figure*}[!tb]
    \centering
    \includegraphics[width=0.95\textwidth,height=0.95
     \textheight,keepaspectratio]{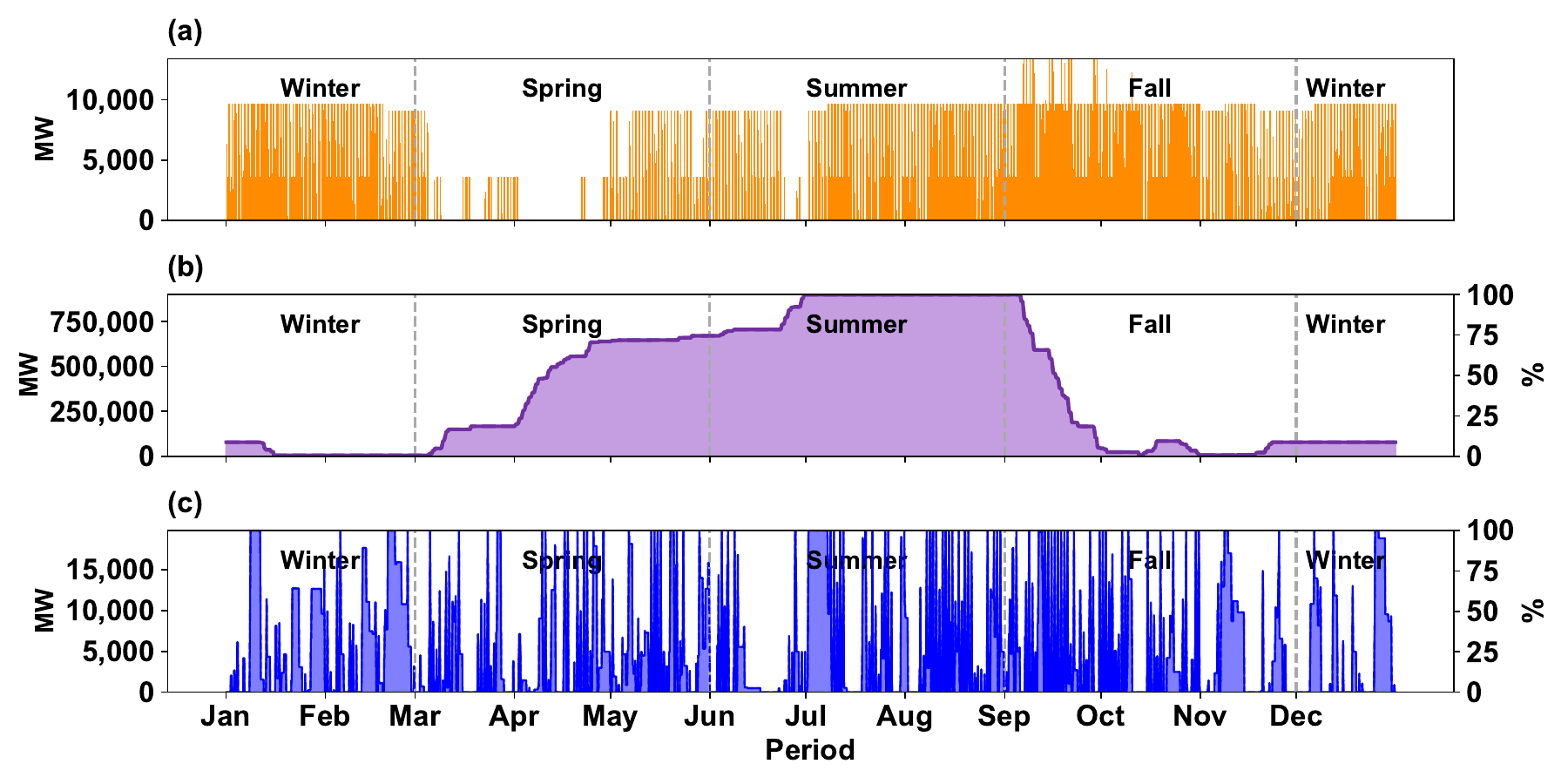}
    \caption{(a) Aggregate hourly dispatch of all gas power plants from the {\it baseline model}'s solution. (b) 100-hour LDES system with 9 GW of power capacity -- state of charge according to the {\it opportunity value maximization model}'s solution. (c) 4-hour SDES system with 5 GW of power capacity -- state of charge according to the {\it opportunity value maximization model}'s solution.}
    \label{Fig:Gas_and_LDES_SDES}
\end{figure*}

For each considered amount of 100-hour LDES power capacity, a different mix of renewable energy investment is required, as shown in Fig. \ref{Fig:Additional_investment}. The candidate options in the {\it opportunity value maximization model} are solar, wind-ons (onshore), wind-ofs (offshore), and SDES (4h duration) assets. From these options, wind-ofs is never chosen by the model, and SDES investment is only required for LDES power capacities lower than 14GW. In general, as LDES power capacity increases, the amount of additional installed capacity from other technologies decreases, with the exception of LDES power capacities ranging between 15 and 20 GW. Within this specific range, LDES provides enough support for the system to choose to invest primarily in solar, which is the less expensive renewable technology in this case study but it is not available at all hours of the day. Furthermore, we observe a saturation in the contribution of LDES to decrease renewable investments when its power capacity reaches 60 and 70 GW. In this case, an amount of 13,710 MW (mix of solar and wind-ons) of newly included renewables is the minimum requirement to retire gas units and use LDES as a firm energy source.

Fig. \ref{Fig:Cost_Reduction} depicts the the reduction in overall system costs as we increase the LDES power capacity once gas units are retired. Essentially, without LDES, the retirement of gas power plants would require substantial investments in SDES and renewable generators. These investments are significantly reduced as we increase the LDES power capacity present in the system. For instance, when the LDES power capacity is 8.7 GW, there is an overall {\color{black}annual system cost reduction} of 7.67 million dollars compared to the costs obtained from the {\it baseline model}. This reduction justifies an investment cost of {\color{black}US\$13.74 $\text{kW}^{-1}$} for LDES. As we increase the LDES power capacity to 17 GW, the overall {\color{black}annual system cost reduction} reaches 558.93 million dollars which can allow the investment cost of LDES to be {\color{black}US\$512.54 $\text{kW}^{-1}$}.

\begin{figure}[!tb]
    \centering
    \includegraphics[width=.45\textwidth,height=0.40
     \textheight,keepaspectratio]{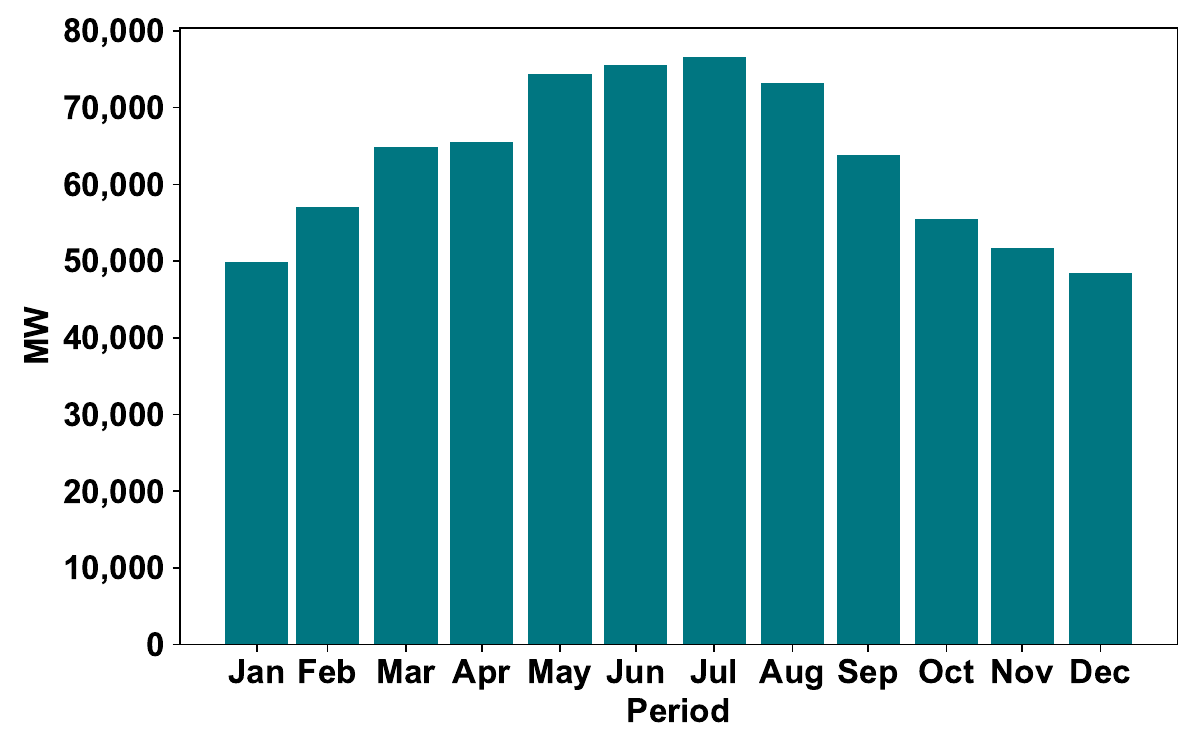}
    \caption{
    {\color{black}Average hourly availability per month of all installed renewable sources in California per month in 2050 according to the solution of the {\it opportunity value maximization model} when the 100-hour LDES power capacity is 9 GW.}
    }
    \vspace{-0.2cm}
    \label{Fig:Availability_Renewable}
\end{figure}

\begin{figure}[!tb]
    \centering
    \includegraphics[width=.4\textwidth,height=0.35\textheight,keepaspectratio]{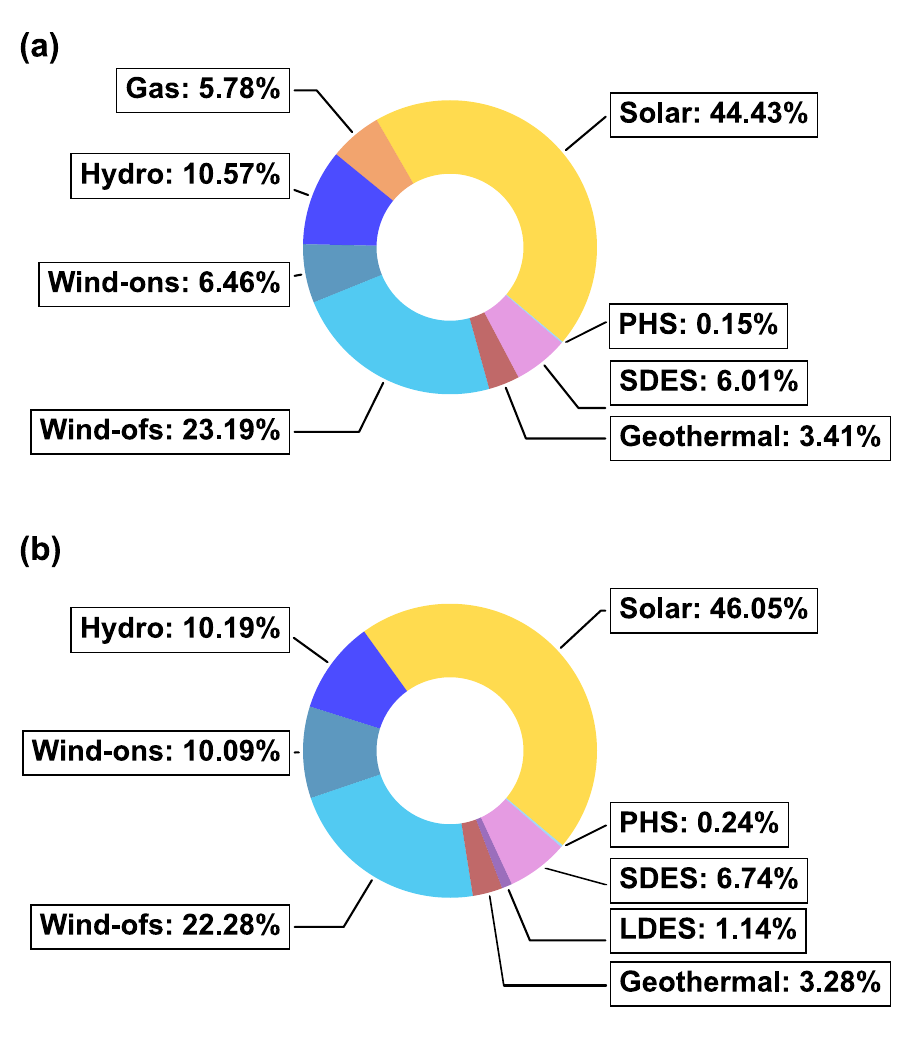}
    \caption{Annual generation contribution per technology for California in 2050 -- (a) results of the {\it baseline model} and (b) results of the {\it opportunity value maximization model} {\color{black}when the 100-hour LDES power capacity is 70 GW}.}
    \vspace{-0.4cm}
    \label{Fig:Generation}
\end{figure}

\begin{figure}[!tb]
    \centering
    \includegraphics[width=.4\textwidth,height=0.35
     \textheight,keepaspectratio]{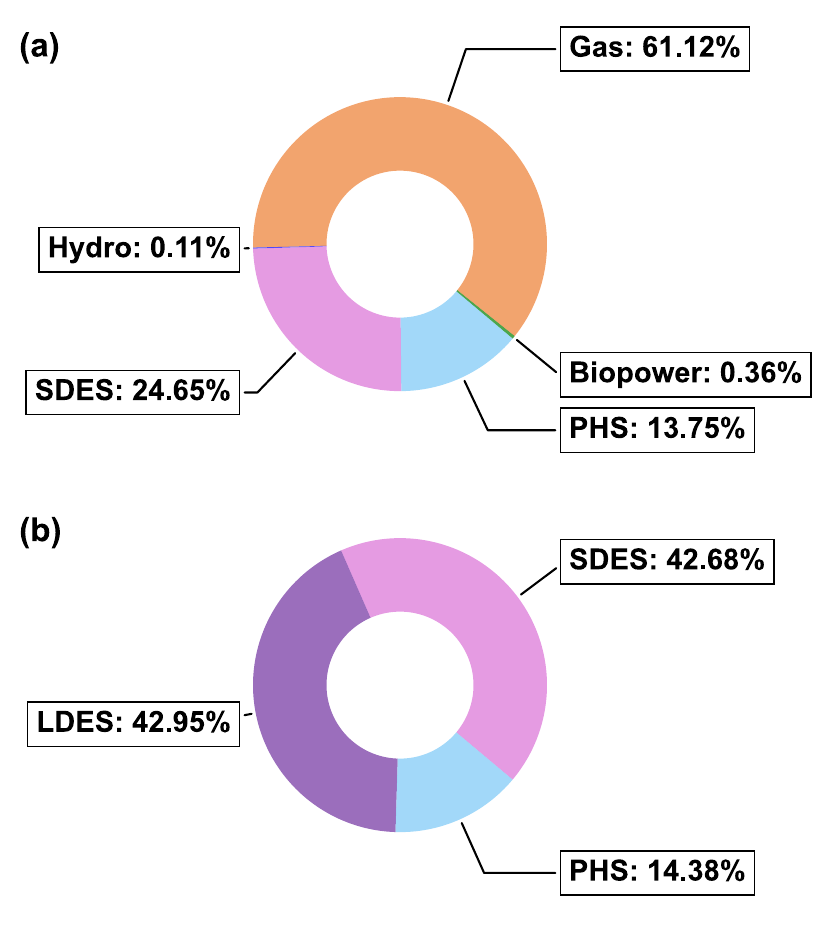}
    \caption{Annual reserve provision per technology for California in 2050 -- (a) results of the {\it baseline model} and (b) results of the {\it opportunity value maximization model} {\color{black}when the 100-hour LDES power capacity is 70 GW}.}
    \vspace{-0.2cm}
    \label{Fig:Reserves}
\end{figure}

\subsection{Operational results}

To justify the previously mentioned boundary costs, the LDES provides the system with an instrumental operational contribution. For example, Fig. \ref{Fig:Gas_and_LDES_SDES} (a) depicts the aggregate dispatch of gas power plants during a year according to the {\it baseline model} whereas Figs. \ref{Fig:Gas_and_LDES_SDES} (b) and \ref{Fig:Gas_and_LDES_SDES} (c) display the state of charge of LDES and SDES systems, respectively, when the 100-hour LDES power capacity is 9 GW. In addition, Fig. \ref{Fig:Availability_Renewable} shows the {\color{black}the average hourly renewable availability per month in 2050 in California once all the investments indicated by the {\it opportunity value maximization model} take place}. As can be seen in Fig. \ref{Fig:Availability_Renewable}, renewables in California, in general, can provide significantly higher levels of generation output during spring and summer compared to fall and winter. In this context, on the one hand, the participation of SDES systems is very important to balance intraday fluctuations of renewables as indicated by the high frequency in the change of their aggregate state of charge in Fig. \ref{Fig:Gas_and_LDES_SDES} (c). On the other hand, SDES systems are not able to help the system to cope with seasonal changes in renewable generation patterns. In the reference system tested through the {\it baseline model}, gas units provide the firm generation needed especially during September when there is a sharp decline in renewable generation output. In the absence of gas plants, LDES takes advantage of the excess of renewable availability during spring and summer to fully charge so as to contribute as a firm resource afterward.

Finally, the overall annual contribution in terms of generation output and reserve provision from each technology is illustrated in Figs. \ref{Fig:Generation} and \ref{Fig:Reserves}, respectively, according to the results of (a) the {\it baseline model} and (b) the {\it opportunity value maximization model}. In both figures, the participation of LDES and SDES as well as the increased renewable generation compensate the retirement of gas 

\vspace{-0.3cm}
\section{Conclusion}\label{sec:Conclusions}
\vspace{-0.1cm}
In this paper, we proposed a systematic manner to estimate the boundary costs of LDES and provided a realistic case study on the role and costs of LDES to fully decarbonize the California's energy system in 2050. The proposed methodology consists of a baseline model and an opportunity value maximization model, which are solved sequentially to assess the maximum reduction in overall costs achievable through a specified quantity of LDES in terms of energy and power. Our results demonstrated that full decarbonization in California in 2050 can be cost-effectively achieved if at least 8.7 GW power capacity of a 100-h LDES is present. At this quantity, if the investment cost of LDES is {\color{black}US\$13.74 $\text{kW}^{-1}$}, the overall annual system costs of the fully decarbonized system will be the same as the reference system, which still relies on gas units to provide firm generation. Furthermore, we observed that the boundary cost of LDES raises when the power capacity ranges from 8.7 GW to 17 GW. This increase is due to the higher rate of reduction in system costs per additional kW of LDES power capacity within this interval. Once the LDES power capacity is greater than 17 GW, the rate of reduction declines and the boundary cost of LDES follows the same pattern.

\bibliographystyle{IEEEtran}
\bibliography{IEEEabrv,References}

\end{document}